\documentclass{amsart}

\usepackage{amsmath,amssymb,enumerate}

\newtheorem{theorem}{Theorem}
\newtheorem{proposition}[theorem]{Proposition}
\newtheorem{lemma}[theorem]{Lemma}
\newtheorem{corollary}[theorem]{Corollary}

\newcommand{\A}{\mathcal{A}}
\newcommand{\B}{\mathcal{B}}

\newcommand{\E}{\mathcal{E}}
\newcommand{\F}{\mathcal{F}}
\newcommand{\G}{\mathcal{G}}
\renewcommand{\H}{\mathcal{H}}
\newcommand{\K}{\mathcal{K}}
\renewcommand{\L}{\mathcal{L}} 
\newcommand{\M}{\mathcal{M}}
\newcommand{\N}{\mathbb{N}}
\renewcommand{\P}{\mathcal{P}}
\newcommand{\R}{\mathbb{R}}
\renewcommand{\r}{\kern-2pt\upharpoonright\kern-2pt}

\newcommand{\X}{\mathcal{X}}

\newcommand{\CL}{\kern -1pt\mathit{CL}}
\newcommand{\size}[1]{\left|#1\right|}

\DeclareMathOperator{\cl}{cl}
\DeclareMathOperator{\diam}{diam}
\DeclareMathOperator{\Eq}{\mathit{Eq}}
\let\int\relax
\DeclareMathOperator{\int}{int}
\DeclareMathOperator{\osc}{osc}
\DeclareMathOperator{\pr}{pr}
\DeclareMathOperator{\rng}{rng}

\begin{document}

\title[A Galois connection related to restrictions]{A Galois connection related to restrictions of continuous real functions}
\thanks{This work was supported by the grant of the Slovak Scientific Grant Agency [VEGA 1/0097/16].}
\author{Peter Elia\v s}
\address{Mathematical Institute, Slovak Academy of Sciences, Gre\v s\'akova~6, 040~01 Ko\v sice, Slovakia}
\email{elias@saske.sk}

\date{}

\begin{abstract}
Given a family of continuous real functions $\mathcal{G}$, let $R_\mathcal{G}$ be a binary relation defined as follows: a continuous function $f\colon\mathbb{R}\to\mathbb{R}$ is in the relation with a closed set $E\subseteq\mathbb{R}$ if and only if there exists $g\in\mathcal{G}$ such that $f\upharpoonright E = g\upharpoonright E$.
We consider a Galois connection between families of continuous functions and hereditary families of closed sets of reals naturally associated to $R_\mathcal{G}$.
We study complete lattices determined by this connection and prove several results showing the dependence
of the properties of these lattices on the properties of $\mathcal{G}$.
In some special cases we obtain exact description of these lattices.
\end{abstract}



\maketitle

\section{Introduction and statement of main results}

Let $X$ and $Y$ be topological spaces.
Denote by $\CL(X)$ the family of all closed subsets of $X$,
and by $C(X,Y)$ the family of all continuous functions $f\colon X\to Y$.
Given a fixed family $\G\subseteq C(X,Y)$, let $R_\G$ be the binary relation defined by
$$R_\G=\{(f,E)\in C(X,Y)\times\CL(X)\!:(\exists g\in\G)\,f\r E=g\r E\}.$$
For $\F\subseteq C(X,Y)$ and $\E\subseteq\CL(X)$ denote
\begin{align*}
  &E_\G(\F)=\{E\in\CL(X)\!:(\forall f\in\F)\,(f,E)\in R_\G\},\\
  &F_\G(\E)=\{f\in C(X,Y)\!:(\forall E\in\E)\,(f,E)\in R_\G\}.
\end{align*}
The mappings
$$E_\G\colon\P(C(X,Y))\to\P(\CL(X))\quad\text{and}\quad F_\G\colon\P(\CL(X))\to\P(C(X,Y))$$
form a Galois connection between the partially ordered sets $(\P(C(X,Y)),\allowbreak\subseteq)$ and $(\P(\CL(X)),\subseteq)$.
This means that $E_\G$ and $F_\G$ are inclusion-reversing mappings such that for any $\F\subseteq C(X,Y)$ and $\E\subseteq\CL(X)$
one has $$\E\subseteq E_\G(\F)\quad\text{if and only if}\quad F_\G(E)\supseteq\F.$$
The compond mappings $$E_\G F_\G\colon\P(\CL(X))\to\P(\CL(X))\quad\text{and}\quad F_\G E_\G\colon\P(C(X,Y))\to\P(C(X,Y))$$
are closure operators on $\P(\CL(X))$ and $\P(C(X,Y))$, respectively.
Hence, for any $\E\subseteq\CL(X)$, $\E\subseteq E_\G F_\G(\E)$ and $E_\G F_\G(\E)=\E$ if and only if $\E=E_\G(\F)$ for some $\F\subseteq C(X,Y)$.
Similarly, for any $\F\subseteq C(X,Y)$, $\F\subseteq F_\G E_\G(\F)$ and $F_\G E_\G(\F)=\F$ if and only if $\F=F_\G(\E)$ for some $\E\subseteq\CL(X)$.
Let us denote
\begin{align*}
  &\K_\G=\{\E\subseteq\CL(X)\!:E_\G F_\G(\E)=\E\}=\{E_\G(\F)\!:\F\subseteq C(X,Y)\},\\
  &\L_\G=\{\F\subseteq C(X,Y)\!:F_\G E_\G(\F)=\F\}=\{F_\G(\E)\!:\E\subseteq\CL(X)\}
\end{align*}
the classes of all closed families with respect to the closure operators associated with the relation $R_\G$.
The families $\K_\G$ and $\L_\G$ are in a one-to-one correspondence and, when ordered by inclusion, form dually isomorphic complete lattices.
In fact, the mapping $\E\mapsto F_\G(\E)$ is an isomorphism $(\K_\G,\subseteq)\to(\L_\G,\supseteq)$ and its inverse is the mapping
$\F\mapsto E_\G(\F)$.
Moreover, the infimum in both lattices coincides with the set-theoretic intersection.
Let us note that in a complete lattice there exist the least and the greatest elements.

For a history of Galois connections and their applications we refer the reader to \cite{Erne}.
For more on their relations to complete lattices and formal concept analysis see \cite{Davey-Priestley}.
Let us note that Galois connections occur naturally in various settings; for some examples
related to analysis and topology see \cite{Lillemets} or \cite{Szasz}.
For applications of Galois connections in the theory of cardinal characteristics see \cite{Blass}.
Our study of restrictions of continuous functions was loosely motivated by classical
notions of Kronecker and Dirichlet sets from harmonic analysis, see, \cite{Rudin} and \cite{Korner}.

In the present paper we deal with the case $X=Y=\R$.
Our aim is to analyze the structure of the lattices $\K_\G$ and $\L_\G$ for certain simple
families $\G\subseteq C(\R,\R)$.
In Section~2 we characterize the elements of the lattices $\K_\G$ and $\L_\G$.
We prove that every element of $\K_\G$ is a hereditary family of closed sets and that each
hereditary family of closed sets is the least element of some lattice $\K_\G$.
We also find a family $\G$ such that $\K_\G$ is the lattice of all nonempty hereditary families
of closed sets.
In Sections 3--5 we describe the lattice $\K_\G$ for three families $\G$ determined
by a single continuous function $g$: the singleton $\{g\}$,
the family of all functions $f$ such that $f(x)<g(x)$ for all $x$,
and the family of all functions $f$ satisfying $f(x)\neq g(x)$ for all $x$.
In each case we characterize all families that yield the same lattice $\K_\G$.

\subsection{Notation and terminology.}
For $\F\subseteq C(\R,\R)$ and $E\subseteq\R$ we denote $\F\r E=\{f\r E\!:f\in\F\}$.
For $x\in\R$ we also denote $\F[x]=\{f(x)\!:f\in\F\}$.
If $\H\subseteq C(E,\R)$, we denote $[\H]=\{f\in C(\R,\R)\!:f\r E\in\H\}$.
We write $[h]$ instead of $[\{h\}]$ for $h\in C(\R,\R)$.

For $E\subseteq\R$ let $\CL(E)$ denote the family of all closed subsets of $E$.
To avoid ambiguity, we use notation $\CL(E)$ only if $E$ is closed; otherwise the family
of all subsets of $E$ that are closed in $\R$ is expressed by the term $\CL(\R)\cap\P(E)$.
Denote $\Eq_{f,g}=\{x\in\R\!:f(x)=g(x)\}$ for $f,g\in C(\R,\R)$.
Then for any $\F\subseteq C(\R,\R)$ and $\E\subseteq\CL(\R)$ we have
$E_\G(\F)=\bigcap_{f\in\F}\bigcup_{g\in\G}\CL(\Eq_{f,g})$ and
$\F_\G(\E)=\bigcap_{E\in\E}\bigcup_{g\in\G}[g\r E]$.

We shall identify a function $f\colon\R\to\R$ and its graph $\{(x,y)\in\R^2\!:f(x)=y\}$.
We say that a family $\G\subseteq C(\R,\R)$ is \emph{complete} if $g\in\G$
holds for every $g\in C(\R,\R)$ satisfying $g\subseteq\bigcup\G$.
A family $\G$ is \emph{connected} if for any $f,g\in\G$ and $x\neq y$ there exists
$h\in\G$ such that $h(x)=f(x)$ and $h(y)=g(y)$.

Let $\R^*=\R\cup\{-\infty,\infty\}$.
For $f,g\in C(\R,\R^*)$, define inequalities $f<g$ and $f\le g$
by $(\forall x\in\R)\,f(x)<g(x)$ and $(\forall x\in\R)\,f(x)\le g(x)$, respectively.
Further, let $(f,g)=\{h\in C(\R,\R)\!:f<h<g\}$ and
$[f,g\kern.5pt]=\{h\in C(\R,\R)\!:f\le h\le g\}$.
If there is no ambiguity we denote the constant function $f\colon x\mapsto z\in C(\R,\R^*)$ simply by $z$.

Let $\X$ be a family of subsets of a topological space.
We say that $\X$ is \emph{separated} if for any distinct sets $X,Y\in\X$
one can find disjoint open sets $U,V$ such that $X\subseteq U$, $Y\subseteq V$
and $(\forall Z\in\X)\ Z\subseteq U\,\lor\,Z\subseteq V$.

\subsection{Main results}
We will prove the following equalities and inclusions.

\begin{theorem}
Let $g\in C(\R,\R)$ and let\/ $\G\subseteq C(\R,\R)$ be nonempty.
\begin{enumerate}[\rm (3b)]
  \item[\rm (1a)] $\K_{\{g\}}=\{\CL(E)\!:E\in\CL(\R)\}$.
  \item[\rm (1b)] $\K_{\{g\}}\subseteq\K_\G$ if and only if\/ $\G[x]\neq\R$ for every $x\in\R$.
  \item[\rm (1c)] $\K_{\{g\}}\supseteq\K_\G$ if and only if\/ $\G=[f,h\kern1pt]$
    for some $f,h\in C(\R,\R^*)$.
  \smallskip
  \item[\rm (2a)] $\K_{(g,\infty)}=\{\CL(\R)\cap\P(X)\!:X\subseteq\R\}$.
  \item[\rm (2b)] $\K_{(g,\infty)}\subseteq\K_\G$ if and only if for every $x\in\R$ there exists $f\in C(\R,\R)$ such that
    $E_\G(\{f\})=\{E\in\CL(\R)\!:x\notin E\}$.
  \item[\rm (2c)] $\K_{(g,\infty)}\supseteq\K_\G$ if and only if\/ $\G$ is complete and connected.
  \smallskip
  \item[\rm (3a)] $\K_{(-\infty,g)\cup(g,\infty)}=\big\{\CL(\R)\cap\bigcup_{X\in\X}\!:\X\subseteq\P(\R)\text{ is separated\/}\big\}$.
  \item[\rm (3b)] $\K_{(-\infty,g)\cup(g,\infty)}\subseteq\K_\G$ if and only if
    $\CL(\R)\cap\P(\R\setminus\{x\})\in\K_\G$ for every $x\in\R$
    and\/ $\CL(\R)\cap(\P(U)\cup\P(\R\setminus\cl U))\in\K_\G$ for every regular open set $U\subseteq\R$.
  \item[\rm (3c)] $\K_{(-\infty,g)\cup(g,\infty)}\supseteq\K_\G$ if and only if\/ $\G=\bigcup_{i\in I}\G_i$
    for some linearly ordered set $(I,<)$ and an indexed system
    of complete connected families $\{\G_i\!:i\in I\}$ 
    such that for every $i\in I$ there exist functions
    $f_i,h_i\in C(\R,\R^*)$ satisfying
    $$\bigcup_{j<i}\G_j\subseteq(-\infty,f_i),\quad
    \G_i\subseteq(f_i,h_i)\quad\text{and}\quad
    \bigcup_{j>i}\G_j\subseteq(h_i,\infty).$$
\end{enumerate}
\end{theorem}

First three statements are proved in Section~3 (Theorems~\ref{thm-singleton}, \ref{thm-contains-all-C(E)} and \ref{thm-is-contained-in-all-C(E)}),
statements (2a)--(2c) in Section~4 (Theorems~\ref{thm-strictly-below-1}, \ref{thm-strictly-below-equiv-1} and \ref{thm-strictly-below-equiv-2}),
and (3a)--(3c) in Section~5 (Theorems~\ref{thm-diff-equal}, \ref{thm-diff-2} and \ref{thm-diff-3}).

\section{The elements of lattices $\K_\G$ and $\L_G$}

We begin with two extremal cases.

\begin{proposition} \label{prop-1}
Let $\G\subseteq C(\R,\R)$.
\begin{enumerate}[\rm (1)]
  \item $\K_\emptyset=\{\emptyset,\CL(\R)\}$, $\L_\emptyset=\{\emptyset,C(\R,\R)\}$.
  \item $\K_{C(\R,\R)}=\{\CL(\R)\}$, $\L_{C(\R,\R)}=\{C(\R,\R)\}$.
\end{enumerate}
\end{proposition}

If $\G\neq\emptyset$ then every family $\E\in\K_\G$ (as well as every $\F\in\L_\G$)
is nonempty because it contains $\emptyset$.
Hence, $\emptyset\in\K_\G$ if and only if $\emptyset\in\L_\G$ if and only if $\G=\emptyset$.

\begin{proposition} \label{prop-2}
Let $\G\subseteq C(\R,\R)$, $\G\neq\emptyset$.
\begin{enumerate}[\rm (1)]
  \item The least element of\/ $\K_\G$ is $E_\G(C(\R,\R))=\{E\in \CL(\R)\!:\G\r E=C(E,\R)\}$.
  \item The least element of $\L_\G$ is $F_\G(\CL(\R))=F_\G(\{\R\})=\G$.
\end{enumerate}
\end{proposition}

It follows that the lattices $\K_\G$ and $\L_\G$ have at least two elements if and only if
$\G\neq C(\R,\R)$.

By Proposition~\ref{prop-2} (2), every family $\G\subseteq C(\R,\R)$ is the least element of $\L_\G$.
Now we are going to characterize families $\E\subseteq\CL(\R)$ that can be least elements of lattices
$\K_\G$.

We say that a family $\E\subseteq\CL(\R)$ is \emph{hereditary} if for any $D,E\in\CL(\R)$,
if $D\subseteq E$ and $E\in\E$ then $D\in\E$.
We show that the elements of lattices $\K_\G$ are exactly hereditary subfamilies of $\CL(\R)$
and every hereditary family $\E\subseteq\CL(\R)$ is the least element of some lattice $\K_\G$.

\begin{proposition} \label{prop-hereditary}
Let $\E\subseteq \CL(\R)$.
The following conditions are equivalent.
\begin{enumerate}[\rm (1)]
  \item There exists $\G\subseteq C(\R,\R)$ such that $\E\in \K_\G$.
  \item $\E$ is hereditary.
\end{enumerate}
\end{proposition}

\begin{proof}
$(1) \Rightarrow (2)$. If $\E\in \K_\G$ then $\E=E_\G(\F)$ for some $\F\subseteq C(\R,\R)$.
It follows from the definition that the family $E_\G(\F)$ is hereditary.

$(2) \Rightarrow (1)$. Fix $f\in C(\R,\R)$.
For every $E\in\E$, fix $g_E\in C(\R,\R)$ such that $\Eq_{f,g_E}=E$ and
let $\G=\{g_E\!:E\in\E\}$.
Then $\E=E_\G(\{f\})\in\K_\G$.
\end{proof}

\begin{lemma} \label{lem-incr-bij}
Let $I,J\subseteq\R$ be non-degenerate, bounded, closed intervals.
For every $n\in\omega$, let $x_n\in\int I$ be distinct and $A_n\subseteq J$ be dense in $J$.
Then there exists an increasing bijection $f\colon I\to J$ such that $(\forall n\in\omega)\,f(x_n)\in A_n$.
\end{lemma}

\begin{proof}
Let us note that any increasing bijection from $I$ to $J$ is necessarily continuous.
Define a sequence of increasing bijections $f_n\colon I\to J$ by induction as follows.
Let $f_0$ be linear.
For every $n$, let $a_0,\dots,a_{n+1}$ be the increasing enumeration of the set
$\{a,b\}\cup\{x_j\!:j<n\}$, where $I=[a,b]$.
Assume that $f_n\colon I\to J$ is an increasing bijection
which is linear on each interval $[a_j,a_{j+1}]$ and moreover the function $f_n-\frac{1}{2}f_0$
is strictly inceasing.
For every $j\le n$, if $x_n\notin(a_j,a_{j+1})$ then
let $f_{n+1}(x)=f_n(x)$ for every $x\in[a_j,a_{j+1}]$.
If $x_n\in(a_j,a_{j+1})$, let $f_{n+1}$ be defined linearly on intervals
$[a_j,x_n]$ and $[x_n,a_{j+1}]$, where $f_{n+1}(a_j)=f_n(a_j)$, $f_{n+1}(a_{j+1})=f_n(a_{j+1})$,
and $f_{n+1}(x_n)\in A_n$ is chosen so that
$f_{n+1}-\frac{1}{2}f_0$ is strictly increasing and for every $x\in(a_j,a_{j+1})$,
$\size{f_{n+1}(x)-f_n(x)}<2^{-n}$.
We obtain a uniformly convergent sequence of increasing bijections $f_n\colon I\to J$.
Its limit $f\colon I\to J$ is continuous and surjective.
Function $f-\frac{1}{2}f_0$, being a limit of a sequence of strictly increasing functions, is non-decreasing,
hence $f$ is strictly increasing.
For every $n$ we have $f(x_n)=f_{n+1}(x_n)\in A_n$.
\end{proof}

\begin{theorem}
Let $\E\subseteq\CL(\R)$ be hereditary.
Then there exists $\G\subseteq C(\R,\R)$ such that $\E$ is the least element of\/ $\K_\G$.
\end{theorem}

\begin{proof}
Fix disjoint countable dense sets $D_0,D_1\subseteq\R$ and
$h\colon\R\to\{0,1\}$ such that both $h^{-1}[\{0\}]$ and $h^{-1}[\{1\}]$ are dense.
Given a hereditary family $\E\subseteq\CL(\R)$, let $\G$ be
the family of all functions $g\in C(\R,\R)$ such that
$$(\forall E\in\CL(\R)\setminus\E)(\exists U\text{ open, }E\cap U\neq\emptyset)
(\forall x\in E\cap U)\ g(x)\notin D_{h(x)}.$$
We prove that $\E=E_\G(C(\R,\R))$.

Let us first show that for every $E\in\E$ and $f\in C(\R,\R)$ there exists $g\in[f\r E]$
such that $g(x)\notin D_{h(x)}$ for all $x\notin E$.
Without a loss of generality we may assume that the complement of $E$ is a disjoint union of
bounded open intervals and that the values of $f$ at the endpoints
of each of these intervals are different.
We can accomplish this by adding to $E$ an unbounded discrete set $Z\subseteq\R\setminus E$
dividing each interval adjacent to $E$, and suitably modifying the values of $f$ outside $E$
to ensure that $f(z)\notin D_{h(z)}$ for $z\in Z$ and $f(a)\neq f(b)$ for each interval
$[a,b]$ adjacent to $E\cup Z$.

Let $[a,b]$ be a closed interval adjacent to $E$.
Assume that $f(a)<f(b)$.
Let $I=[f(a),f(b)]$, $J=[a,b]$, $\{x_n\!:n\in\omega\}=(D_0\cup D_1)\cap\int I$, and for every $n$, let
$A_n=J\cap h^{-1}[\{i\}]$ where $i\in\{0,1\}$ is such that $x_n\notin D_i$.
Let $g_J\colon I\to J$ be the increasing bijection obtained in Lemma~\ref{lem-incr-bij}.
Its inverse $g_J^{-1}\colon [a,b]\to [f(a),f(b)]$ is an increasing bijection as well
and for every $x\in (a,b)$ we have $g_J^{-1}(x)\notin D_{h(x)}$.
Similarly, if $f(b)<f(a)$ then there exists a decreasing bijection
$g_J^{-1}\colon [a,b]\to [f(b),f(a)]$
such that $g_J^{-1}(x)\notin D_{h(x)}$ for all $x\in (a,b)$.
Define $g\colon\R\to\R$ by
$$g(x)=\begin{cases}f(x),&\text{if $x\in E$},\\
g^{-1}_J(x)&\text{if $J$ is a closed interval adjacent to $E$ and $x\in\int J$}.\end{cases}$$
Obviously, $g$ is continuous and $g(x)\notin D_{h(x)}$ for $x\notin E$.
We show that $g\in\G$.
Indeed, if $D\in\CL(\R)\setminus\E$ then $D\nsubseteq E$, since $\E$ is hereditary.
Let $U=\R\setminus D$.
Then $U$ is open, $E\cap U\neq\emptyset$, and $g(x)\notin D_{h(x)}$ for every $x\in E\cap U$.

We have shown that for every $E\in\E$ and $f\in C(\R,\R)$ there exists $g\in\G$ such that $g\r E=f\r E$,
hence $\E\subseteq E_\G(C(\R,\R))$.
To prove the opposite inclusion, let us take $E\in\CL(\R)\setminus\E$.
Let $\{x_n\!:n\in\omega\}$ be a countable dense subset of $E$, and
for every $n$, let $A_n=D_{h(x_n)}$.
By repeated use of Lemma~\ref{lem-incr-bij} we can find $f\in C(\R,\R)$ such that for $f(x_n)\in A_n$ for all $n$.
We show that $f\r E\notin\G\r E$ and hence $E\notin E_\G(C(\R,\R))$.

Let $g\in\G$ be arbitrary.
There exists an open set $U$ such that $E\cap U\neq\emptyset$ and $g(x)\notin D_{h(x)}$ for every $x\in E\cap U$.
Find $n$ such that $x_n\in E\cap U$.
Then $f(x_n)\in D_{h(x_n)}$ while $f(x_n)\neq g(x_n)$, hence $f\r E\neq g\r E$.
\end{proof}

\begin{theorem}
There exists $\G\subseteq C(\R,\R)$ such that\/ $\K_\G$ contains every nonempty hereditary family
$\E\subseteq\CL(\R)$.
\end{theorem}

\begin{proof}
Let $\{E_\alpha\!:\alpha<2^\omega\}$
be a one-to-one enumeration of all nonempty closed subsets of $\R$.
For each $\alpha<2^\omega$ fix some $x_\alpha\in E_\alpha$.
Using transfinite induction for $\alpha<2^\omega$ we define $y_\alpha\in\R$
and a sequence of functions $\{g_{\alpha,n}\!:n\in\omega\}\subseteq C(\R,\R)$.

We proceed as follows.
If $y_\beta$ and $g_{\beta,n}$ are defined for all $\beta<\alpha$ and $n\in\omega$,
find $y_\alpha\notin\{y_\beta\!:\beta<\alpha\}\cup\{g_{\beta,n}(x_\alpha)\!:\beta<\alpha,\,n\in\omega\}$.
Let $\{I_{\alpha,n}\!:n\in\omega\}$ be the family of all nonempty open intervals with rational endpoints
having nonempty intersection with $E_\alpha$.
For every $n$, there exists a function $g_{\alpha,n}\in C(\R,\R)$ such that
$g_{\alpha,n}(x)=y_\alpha$ if and only if $x\notin I_{\alpha,n}$,
and $g_{\alpha,n}(x_\beta)\neq y_\beta$ for all $\beta<\alpha$.

Let $\G=\{g_{\alpha,n}\!:\alpha<2^\omega,\,n\in\omega\}$.
For every $\alpha<2^\omega$, let $f_\alpha$ be the constant function with value $y_\alpha$.
We show that $f_\alpha\r E_\alpha\notin\G\r E_\alpha$.
If $g\in\G$ then $g=g_{\beta,n}$ for some $\beta<2^\omega$ and $n\in\omega$.
If $\beta<\alpha$ then $g_{\beta,n}(x_\alpha)\neq y_\alpha$ by the definition of $y_\alpha$.
Since $x_\alpha\in E_\alpha$, we have $f_\alpha\r E_\alpha\neq g\r E_\alpha$.
If $\beta=\alpha$ then there exists $x\in E_\alpha\cap I_{\alpha,n}$.
We have $g_{\alpha,n}(x)\neq y_\alpha$, hence $f_\alpha\r E_\alpha\neq g\r E_\alpha$.
Finally, if $\beta>\alpha$ then $g_{\beta,n}(x_\alpha)\neq y_\alpha$ by the definition of $g_{\beta,n}$.
Again, $f_\alpha\r E_\alpha\neq g\r E_\alpha$.

Let $\E$ be a nonempty hereditary family of closed subsets of $\R$.
Then $\emptyset\in\E$, hence each $E\in\CL(\R)\setminus\E$ is nonempty.
Denote $\F=\{f_\alpha\!:E_\alpha\in\CL(\R)\setminus\E\}$.
If $E\in\E$ and $f\in\F$ then $f=f_\alpha$ for some $\alpha<2^\omega$ such that $E_\alpha\in\CL(\R)\setminus\E$,
hence $E_\alpha\nsubseteq E$.
There exists $n$ such that $E\subseteq\R\setminus I_{\alpha,n}$.
By the definition of $g_{\alpha,n}$ we have $f_\alpha\r E=g_{\alpha,n}\r E$,
hence $f_\alpha\r E\in\G\r E$.
It follows that $E\in E_\G(\F)$, and we conclude that $\E\subseteq E_\G(\F)$.
Conversely, if $E\in\CL(R)\setminus\E$ then $E=E_\alpha$ for some $\alpha<2^\omega$.
Since $f_\alpha\r E_\alpha\notin\G\r E_\alpha$ and $f_\alpha\in\F$, we have $E\notin E_\G(\F)$.
It follows that $\E=E_\G(\F)$, hence $\E\in\K_\G$.
\end{proof}

\section{Results for family $\G=\{g\}$}

We show that if $\G$ is a singleton then the lattice $\K_\G$ is isomorphic to the complete lattice
$(\CL(\R),\subseteq)$ of all closed subsets of~$\R$.
Let us note that in $(\CL(\R),\subseteq)$ we have $\bigwedge\E=\bigcap\E$ and
$\bigvee\E=\cl\big(\bigcup\E\big)$, for any $\E\subseteq \CL(\R)$.

\begin{theorem} \label{thm-singleton}
Let $\G=\{g\}$, $g\in C(\R,\R)$.
Then $\K_\G=\{\CL(E)\!:E\in \CL(\R)\}$ and $\L_\G=\{[g\r E]\!:E\in \CL(\R)\}$.
\end{theorem}

\begin{proof}
It is clear that $(f,E)\in R_\G$ if and only if $E\subseteq\Eq_{f,g}$,
for any $f\in C(\R,\R)$ and $E\in \CL(\R)$.
Hence, for any $\F\subseteq C(\R,\R)$ we have
\begin{displaymath}
  E_\G(\F)=
  \bigcap_{f\in\F}\CL(\Eq_{f,g})=
  \CL\left(\textstyle\bigcap\nolimits_{f\in\F}\Eq_{f,g}\right).
\end{displaymath}
Since for any set $E\in \CL(\R)$ there exists $f\in C(\R,\R)$ such that $E=\Eq_{f,g}$, we obtain that
$\K_\G=\{E_\G(\F)\!:\F\subseteq C(\R,\R)\}=\{\CL(E)\!:E\in \CL(\R)\}$. 
For any $\E\subseteq \CL(\R)$ we also have
\begin{displaymath}
  F_\G(\E)=
  \bigcap_{E\in\E}[g\r E]=
  \big[g\r \cl\big(\textstyle\bigcup\E\big)\big],
\end{displaymath}
hence $\L_\G=\{F_\G(\E)\!:\E\subseteq \CL(\R)\}=\{[g\r E]\!:E\in \CL(\R)\}$.
\end{proof}

It can be easily seen that if $\G=\{g\}$ then each element $\E\in \K_\G$ can be generated by a family
consisting of a single function $f$:
if $\E\in \K_\G$ then there exists $E\in \CL(\R)$ such that $\E=\CL(E)$,
for any $f\in C(\R,\R)$ satisfying $\Eq_{f,g}=E$ we then have $\E=E_\G(\{f\})$.
Similarly, each $\F\in \L_\G$ can be expressed as
$F_\G(\{E\})$ for some $E\in \CL(\R)$.
Moreover, this set $E$ is unique; if $D,E\in \CL(\R)$ are distinct then
$F_\G(\{D\})\neq F_\G(\{E\})$ by the normality of $\R$.

The next two results allows us to characterize families $\G\subseteq C(\R,\R)$
for which the lattice $\K_\G$ is the same as in Theorem~\ref{thm-singleton}.

\begin{theorem} \label{thm-contains-all-C(E)}
Let $\G\subseteq C(\R,\R)$ be nonempty.
The following conditions are equivalent.
\begin{enumerate}[\rm (1)]
  \item $\{\CL(E)\!:E\in\CL(\R)\}\subseteq\K_\G$.
  \item The least element of $\K_\G$ is $\{\emptyset\}$.
  \item For each $x\in\R$, $\G[x]\neq\R$.
\end{enumerate}
\end{theorem}

\begin{proof}
$(1) \Rightarrow (2)$ is trivial.

$(2) \Rightarrow (3)$.
If $E_\G(C(\R,\R))=\{\emptyset\}$ then for each nonempty $E\in \CL(\R)$ there exists $f\in C(\R,\R)$ such that $f\r E\notin\G\r E$.
In particular, for every $x\in\R$ there exists $y\in\R$ such that
if $f(x)=y$ then $f\r\{x\}\notin\G\r\{x\}$, hence $y\notin\G[x]$.

$(3) \Rightarrow (1)$.
Fix a function $h\in\G$.
For every $E\in \CL(\R)$ and $x\notin E$ let us take $y\notin\G[x]$ and
a function $f_x\in [h\r E]$ such that $f_x(x)=y$.
Let $\F=\{f_x\!:x\notin E\}$.
If $D\in E_\G(\F)$ then for any $x\notin E$ we have $f_x\r D\in\G\r D$, hence $x\notin D$.
It follows that $D\subseteq E$ and thus $E_\G(\F)\subseteq\CL(E)$.
The opposite inclusion is clear, hence we obtain $\CL(E)=E_\G(\F)\in \K_\G$.
\end{proof}

\begin{theorem} \label{thm-is-contained-in-all-C(E)}
Let $\G\subseteq C(\R,\R)$ be nonempty.
The following conditions are equivalent.
\begin{enumerate}[\rm (1)]
  \item $\K_\G\subseteq\{\CL(E)\!:E\in\CL(\R)\}$.
  \item There exist $h_1,h_2\in C(\R,\R^*)$ such that\/
    $\G=\{g\in C(\R,\R)\!:h_1\le g\le h_2\}$.
\end{enumerate}
\end{theorem}

\begin{proof}
$(1) \Rightarrow (2)$.
Denote $H=\bigcup\G$.
Let us first show that $H$ is a closed subset of $\R^2$.
Assume that $(x,y)\in\cl H$.
Since $\G$ is a nonempty family of continuous functions,
there exists in $H$ a sequence of points $\{(x_n,y_n)\!:n\in\omega\}$ converging to $(x,y)$ and such that
all $x_n$ are distinct.
Let $f\in C(\R,\R)$ be such that $f(x_n)=y_n$ for every $n$ and $f(x)=y$.
Then $\E=E_\G(\{f\})$ contains $\{x_n\}$ for every $n$.
Since $\E\in\K_\G$, it follows from 1 that $\cl\{x_n\!:n\in\omega\}\in\E$, hence $\{x\}\in\E$ and so $(x,y)\in H$.

We show that $\G[x]=\{g(x)\!:g\in\G\}$ is connected, for every $x\in\R$.
Otherwise we can find $g_1,g_2\in\G$ and $y\notin\G[x]$ such that $g_1(x)<y<g_2(x)$.
Since $H$ is closed, there exist $a,b,c,d\in\R$ such that $x\in(a,b)$, $y\in (c,d)$, and
$\big((a,b)\times (c,d)\big)\cap H=\emptyset$.
Let $f\in C(\R,\R)$ be such that $f(a)=g_1(a)$ and $f(b)=g_2(b)$.
For $\E=E_\G(\{f\})$ we have $\{a\}\in\E$, $\{b\}\in\E$, hence also $\{a,b\}\in\E$,
and thus there exists $g\in\G$ such that $g(a)=g_1(a)$ and $g(b)=g_2(b)$.
By Intermediate Value Theorem there is $z\in (a,b)$ such that $g(z)\in (c,d)$, which
contradicts the assumption that $\big((a,b)\times(c,d)\big)\cap H$ is empty.
So $\G[x]$ is a connected closed set, that is, a closed interval.

For every $x\in\R$ denote $h_1(x)=\inf\G[x]$ and $h_2(x)=\sup\G[x]$.
We show that $h_1,h_2$ are continuous.
If $y<h_1(x)$ then there exist $a,b,c,d\in\R$ such that $x\in(a,b)$, $y\in(c,d)$, and
$\big((a,b)\times(c,d)\big)\cap H=\emptyset$.
It follows $\big((a,b)\times(-\infty,d)\big)\cap H=\emptyset$,
otherwise one could find a contradiction using Intermediate Value Theorem, as before.
We can conclude that $h_1$ is lower semi-continuous.
Since $h_1$ is the infimum of a family of continuous functions, it is also upper semi-continuous,
and hence continuous.
A similar argument shows the continuity of $h_2$.

It remained to show that $\G=\{g\in C(\R,\R)\!:h_1\le g\le h_2\}$.
The inclusion from left to right is clear.
If $g\in C(\R,\R)$ is such that $h_1\le g\le h_2$, then for every $x\in\R$ we have $g(x)\in\G[x]$,
hence $\{x\}\in\E$ where $\E=E_\G(\{g\})$.
By (1), $\R=\bigcup\E\in\E$, hence $g\in\G$.

$(2) \Rightarrow (1)$.
Let $\E\in\K_\G$, that is, $\E=E_\G(\F)$ for some $\F\subseteq C(\R,\R)$, and let $E=\bigcup\E$.
Then $E=\{x\in\R\!:\F[x]\subseteq\G[x]\}$.
First, let us show that $E\in\CL(\R)$.
If $x\in\cl E$ then there exists a sequence $\{x_n\!:n\in\N\}$ in $E$ such that $x_n\to x$.
For any $y\in\F[x]$, let us take some $f\in\F$ such that $f(x)=y$.
For every $n$ we have $f(x_n)\in\F(x_n)\subseteq\G(x_n)$, hence $h_1(x_n)\le f(x_n)\le h_2(x_n)$.
By the continuity of $f$, $h_1$, and $h_2$ we obtain that $h_1(x)\le f(x)\le h_2(x)$, hence $y\in\G[x]$.
We have $\F[x]\subseteq\G[x]$, hence $x\in E$, and it follows that $E$ is closed.

We show that for every $f\in\F$ there exists $g\in\G$ such that $f\r E=g\r E$.
Fix some $f\in\F$.
For every $x\in E$ we have $f(x)\in\G[x]$, hence $h_1(x)\le f(x)\le h_2(x)$.
Let $g(x)=\min\{\max\{f(x),h_1(x)\},h_2(x)\}$, for all $x\in\R$.
Clearly, $g$ is continuous and $g\r E=f\r E$.
Since $\G$ is nonempty, we have $h_1\le h_2$, and hence also $h_1\le g\le h_2$.
By (2), we have $g\in\G$.
It follows that $E\in\E$, hence $\E=\CL(E)$.
\end{proof}

Now we can characterize those families $\G$ for which
$\K_\G=\{\CL(E)\!:E\in\CL(\R)\}$.
Let us recall that for $f,h\in C(\R,\R^*)$ we denoted $[f,h]=\{g\in C(\R,\R)\!:f\le g\le h\}$,
where $f\le g$ if and only if $(\forall x\in\R)\, f(x)\le g(x)$.

\begin{corollary}
Let $\G\subseteq C(\R,\R)$.
The following conditions are equivalent.
\begin{enumerate}[\rm (1)]
\item $\K_\G=\{\CL(E)\!:E\in\CL(\R)\}$.
\item There exist $f,h\in C(\R,\R^*)$ such that $f\le h$,
$f^{-1}[\R]\cup h^{-1}[\R]=\R$, and $\G=[f,h]$.
\end{enumerate}
\end{corollary}

It follows that the same lattice $\K_\G$ is obtained for families $\G$ of the form
$(-\infty,g\kern.5pt]=\{f\in C(\R,\R)\!:f\le g\}$ and $[\kern.5ptg,\infty)=\{f\in C(\R,\R)\!:g\le f\}$.

\begin{corollary}
Let $g\in C(\R,\R)$.
Then $\K_{(-\infty,g]}=\K_{[g,\infty)}=\{\CL(E)\!:E\in\CL(\R)\}$.
\end{corollary}

\section{Results for family $\G=(g,\infty)$}

Recall that for $f,h\in C(\R,\R^*)$, $(f,h)=\{g\in C(\R,\R)\!:f<g<h\}$
where $f<g$ is a shorthand for $(\forall x\in\R)\,f(x)<g(x)$.

\begin{theorem} \label{thm-strictly-below-1}
Let $\G=(g,\infty)$ where $g\in C(\R,\R)$.
Then $\K_\G=\{\CL(\R)\cap\P(X)\!:X\subseteq\R\}$.
\end{theorem}

\begin{proof}
If $\E\in\K_\G$ then $\E=E_\G(\F)$ for some $\F\subseteq C(\R,\R)$.
Denote $X=\bigcup\E$.
Clearly $\E\subseteq\CL(\R)\cap\P(X)$.
If $E\in\CL(\R)\cap\P(X)$ then for every $x\in E$ we have $x\in D$ for some $D\in\E$,
hence $f(x)>g(x)$ for all $f\in\F$.
For every $f\in\F$ there exists $f'\in\G$ such that $f'\r E=f\r E$;
it suffices to take $f'$ to be linear on each bounded interval adjacent to $E$,
and constant on unbounded ones, if there are any.
It follows that $E\in\E$, and we obtain $\E=\CL(\R)\cap\P(X)$,
hence $\K_\G\subseteq\{\CL(\R)\cap\P(X)\!:X\subseteq\R\}$.

To prove the opposite, let $X\subseteq\R$.
Denote $\F=\{f_a\!:a\in\R\setminus X\}$, where $f_a(x)=g(x)+\size{x-a}$ for all $x\in\R$.
If $E\in E_\G(\F)$ then $f_a(x)>g(x)$ for all $x\in E$ and $a\in\R\setminus X$, hence $E\subseteq X$.
It follows that $E_\G(\F)\subseteq\CL(\R)\cap\P(X)$.
The opposite inclusion is clear since $\F\subseteq F_\G(\CL(\R)\cap\P(X))$.
We obtain that $\CL(\R)\cap\P(X)=E_\G(\F)\in\K_\G$, hence
$\{\CL(\R)\cap\P(X)\!:X\subseteq\R\}\subseteq\K_\G$.
\end{proof}

A similar argument would prove the same result for the family $\G=(-\infty,g)$.
Nevertheless, it will also follow from Corollary~\ref{cor-strictly-below-together} below.

We will characterize those families $\G\subseteq C(\R,\R)$ for which
$\K_\G=\{\CL(\R)\cap\P(X)\!:X\subseteq\R\}$.
Like in the previous section, we characterize both inclusions separately.
For $x\in\R$ denote $\A_x=\CL(\R)\cap\P(\R\setminus\{x\})=\{E\in\CL(\R)\!:x\notin E\}$.

\begin{theorem} \label{thm-strictly-below-equiv-1}
Let $\G\subseteq C(\R,\R)$ be nonempty.
The following conditions are equivalent.
\begin{enumerate}[\rm (1)]
\item $\{\CL(\R)\cap\P(X)\!:X\subseteq\R\}\subseteq\K_\G$.
\item $\{\A_x\!:x\in\R\}\subseteq\K_\G$.
\item For every $x\in\R$ there exists $f\in C(\R,\R)$ such that $E_\G(\{f\})=\A_x$.
\end{enumerate}
\end{theorem}

\begin{proof}
$(1) \Rightarrow (2)$ is clear.

$(2) \Rightarrow (3)$.
For every $x\in\R$, we have $\{x\}\notin\A_x=E_\G(F_\G(\A_x))$,
hence there exists $f\in F_\G(\A_x)$ such that $f\r\{x\}\notin\G\r\{x\}$.
We have $\{f\}\subseteq F_\G(\A_x)$, hence $E_\G(\{f\})\supseteq E_\G(F_\G(\A_x))=\A_x$.
Conversely, if $E\in\CL(\R)\setminus\A_x$ then $x\in E$ and hence $f\r E\notin\G\r E$.
It follows that $E\notin E_\G(\{f\})$, and we obtain $E_\G(\{f\})\subseteq\A_x$.

$(3) \Rightarrow (1)$.
Let $X\subseteq\R$, $\E=\CL(\R)\cap\P(X)$, and $\F=F_\G(\E)$.
We will show that $E_\G(\F)=\E$.
If not, then there exists $E\in E_\G(\F)$ such that $E\nsubseteq X$.
Let $x\in E\setminus X$.
By (3) there exists $f\in C(\R,\R)$ such that $E_\G(\{f\})=\A_x$.
Since $\E\subseteq\A_x$, we have $F_\G(\E)\supseteq F_\G(\A_x)$, hence $f\in F_\G(\E)=\F$.
It follows that $E\in E_\G(\{f\})$, and we come to a contradiction.
\end{proof}

Note that the family $\{\A_x\!:x\in\R\}$ in condition (2) of Theorem~\ref{thm-strictly-below-equiv-1} is minimal.
Given $z\in\R$, let $\G$ be the family of all continuous functions $f$ such that
$f(x)>0$ for all $x\neq z$.
For every $y\in\R$, let $f_y(x)=\size{x-y}$.
If $y\neq z$ then $E_\G(\{f_y\})=\A_y$, hence $\A_y\in\K_\G$.
Since $F_\G(\A_z)=\{f\in C(\R,\R)\!:(\forall x\neq z)\,f(x)>0\}=\G$,
we obtain that $E_\G(F_\G(\A_z))=\CL(\R)$, hence $\A_z\notin\K_\G$.

To characterize all families $\G$ such that $\K_\G\subseteq\{\CL(\R)\cap\P(X)\!:X\subseteq\R\}$
we need the following notion.
We say that a set $H\subseteq\R^2$ is \emph{functionally connected}
if for any two points $(x_1,y_1),(x_2,y_2)\in H$ such that $x_1\neq x_2$,
there exists a continuous function $h\colon[x_1,x_2]\to\R$ such that $h(x_1)=y_1$, $h(x_2)=y_2$,
and the graph of $h$ is included in $H$.
If $H$ is a functionally connected set then $\pr_1[H]$, the projection of $H$ to the first coordinate, is connected.
If $\pr_1[H]$ has at most one point then $H$ is functionally connected.
If $\pr_1[H]$ has more than one point and $H$ is functionally connected then $H$ must be pathwise connected.
A connected set need not to be functionally connected, a simple example is the unit circle $\{(x,y)\!:x^2+y^2=1\}$.

\begin{lemma} \label{lem-func-conn-1}
Let $a<b$ and let $H\subseteq\R^2$ be a functionally connected set such that $[a,b]\subseteq\pr_1[H]$.
Let $h\colon[a,b]\to\R$ be a continuous function such that $h\subseteq H$,
and let $u,v\in\R$ be such that points $(a,u),(b,v)\in H$.
Then for every open interval $J$ such that $u,v\in J$ and\/ $\rng(h)\subseteq J$,
there exists a continuous function $g\colon[a,b]\to J$ such that $g\subseteq H$, $g(a)=u$, and $g(b)=v$.
\end{lemma}

\begin{proof}
Let $a,b,H,h,u,v$, and $J$ be as above.
There exists a continuous function $f\colon[a,b]\to\R$
such that $f\subseteq H$, $f(a)=u$, and $f(b)=v$.
Let $a_2,b_2\in(a,b)$ be such that $a_2<b_2$ and $f(x)\in J$ for every $x\in[a,a_2]\cup[b_2,b]$.

Let us first prove that there exists some $a_1\in[a,a_2]$ and a continuous function
$f_1\colon[a,a_1]\to J$
such that $f_1\subseteq H$, $f_1(a)=f(a)$, and $f_1(a_1)=h(a_1)$.
This is clear if $f(x)=h(x)$ for some $x\in[a,a_2]$.
If this is not the case then the values $f(a)-h(a)$ and $f(a_2)-h(a_2)$ must have the same signs.
Without a loss of generality, assume that $f(a)>h(a)$ and $f(a_2)>h(a_2)$.
Let $f'\colon[a,a_2]\to\R$ be a continuous function such that $f'\subseteq H$,
$f'(a)=f(a)$, and $f'(a_2)=h(a_2)$.
Let $a_0=\max\{x\in[a,a_2]\!:f'(x)=f(x)\}$ and
$a_1=\min\{x\in[a_0,a_2]\!:f'(x)=h(x)\}$.
It follows that $f'(x)\in J$ for every $x\in[a_0,a_1]$, and we can define
$f_1(x)=f(x)$ for $x\in[a,a_0]$ and $f_1(x)=f'(x)$ for $x\in[a_0,a_1]$.
Then $f_1$ is as required.

Similarly, there exist $b_1\in[b',b]$ and a continuous function $f_2\colon[b_1,b]\to J$
such that $f_2\subseteq H$, $f_2(b_1)=h(b_1)$, and $f_2(b)=f(b)$.
Let $g(x)=f_1(x)$ for $x\in[a,a_1]$, $g(x)=h(x)$ for $x\in[a_1,b_1]$, and $g(x)=f_2(x)$ for $x\in[b_1,b]$.
Then $g$ has the required properties.
\end{proof}

\begin{lemma} \label{lem-func-conn-2}
Let $H\subseteq\R^2$ be a functionally connected set such that $\pr_1[H]=\R$.
Then for every $f\in C(\R,\R)$ and $E\in\CL(\R)$ such that $f\r E\subseteq H$,
there exists $g\in C(\R,\R)$ such that $g\r E=f\r E$ and $g\subseteq H$.
\end{lemma}

\begin{proof}
Let $H$, $f$ and $E$ be as assumed.
Let us note that for each point $(x,y)\in H$
there exists $g\in C(\R,\R)$ such that $g(x)=y$ and $g\subseteq H$.

For every closed interval $I$ adjacent to $E$ there exists a continuous function $g_I\colon\R\to\R$
such that $g_I\subseteq H$ and $g_I$ coincides with $f$ at the endpoints of $I$.
Let $g(x)=f(x)$ for $x\in E$ and $g(x)=g_I(x)$ for $x\in I$ if $I$ is a closed interval
adjacent to $E$.
We obtain a function $g\colon\R\to\R$ such that $g\r E=f\r E$ and $g\subseteq H$.

It remains to show that $g_I$ can be chosen so that $g$ is continuous.
This is clear if there are only finitely many such intervals, so we will assume the opposite.
Let $\{I_n\!:n\in\omega\}$ be one-to-one enumeration of all closed intervals adjacent to $E$.
We define intervals $g_n=g_{I_n}$ as follows.

If $I_n$ is unbounded then let $g_n\colon I_n\to\R$ be arbitrary continuous function
such that $g_n\subseteq H$ and $g_n$ coincides with $f$ at the only endpoint of $I_n$.

Assume that $I_n=[a_n,b_n]$.
For every continuous function $h\colon I_n\to\R$ denote $\osc(h)$ its oscillation, that is,
$\osc(h)=\max\{h(x)-h(y)\!:x,y\in I_n\}$.
Let $o_n=\inf\{\osc(h)\!:h\in\H_n\}$, where $\H_n$ is the family of all functions $h\in C(I_n,\R)$
such that $h\subseteq H$ and $h\r\{a_n,b_n\}=f\r\{a_n,b_n\}$.
Choose $g_n\in\H_n$ such that $\osc(g_n)\le o_n + 2^{-n}$.

We will prove that the function $g$ defined above is continuous at every point $z\in\R$.
Let us take a convergent sequence $z_k\to z$.
We will assume that this sequence is increasing, as it suffices to consider only one-sided
limits, and for decreasing sequences the proof is the same.
We may further assume that $z_k\in\R\setminus E$ for all $k$
since we have $g(x)=f(x)$ for $x\in E$ and $f$ is continuous at $z$.
For every $k$, let $n_k$ be such that $z_k\in I_{n_k}$.
If there exist $m,l$ such that $n_k=m$ for all $k>l$, then $g(z_k)=g_m(z_k)$ for all $k>l$, hence
$z\in\cl I_m$ and $g(z_k)\to g(z)$.
So we may assume that $n_k\to\infty$ and $z\in E$.

To prove that $g(z_k)\to g(z)$, it will suffice to show that $\osc(g_{n_k})\to 0$.
Fix some $h\in C(\R,\R)$ such that $h\subseteq H$ and $h(z)=g(z)$.
By Lemma~\ref{lem-func-conn-1}, for every $n$ we have $o_n\le\diam(f[I_n]\cup h[I_n])$.
Since both $f$ and $h$ is continuous at $z$, we have
$\diam(f[I_n]\cup h[I_n])\to 0$, hence
$\osc(g_{n_k})\le o_{n_k}+2^{-n_k}\to 0$.
\end{proof}

Recall that a family $\G\subseteq C(\R,\R)$ is said to be complete if $f\in\G$
for every $f\in C(\R,\R)$ such that $f\subseteq\bigcup\G$.
A complete family $\G\subseteq C(\R,\R)$ is connected if and only if $\bigcup\G$
is functionally connected.

\begin{theorem} \label{thm-strictly-below-equiv-2}
Let $\G\subseteq C(\R,\R)$ be nonempty.
The following conditions are equivalent.
\begin{enumerate}[\rm (1)]
\item $\K_\G\subseteq\{\CL(\R)\cap\P(X)\!:X\subseteq\R\}$.
\item $\G$ is a complete and connected family.
\end{enumerate}
\end{theorem}

\begin{proof}
$(1) \Rightarrow (2)$.
Denote $H=\bigcup\G$.
Clearly, $\G\subseteq\{g\in C(\R,\R)\!:g\subseteq H\}$.
To prove the opposite inclusion, let $g\in C(\R,\R)$ be such that $g\subseteq H$.
Denote $\E=E_\G(\{g\})$.
For every $x\in\R$ we have $g(x)\in\G[x]$, hence $\{x\}\in\E$.
It follows that $\R\in\E$, hence $g\in\G$.
Thus, $\G=\{g\in C(\R,\R)\!:g\subseteq H\}$, hence $\G$ is complete.

It remains to show that $H$ is functionally connected.
Let $(x_1,y_1),(x_2,y_2)\in H$ and $x_1<x_2$.
Let $f\in C(\R,\R)$ be such that $f(x_1)=y_1$ and $f(x_2)=y_2$, and let $\E=E_\G(\{f\})$.
Since $\{x_1\},\{x_2\}\in\E$ and $\E\in\K_\G$, it follows that $\{x_1,x_2\}\in\E$.
Hence, there exists $g\in\G$ such that $g(x_1)=y_1$ and $g(x_2)=y_2$.

$(2) \Rightarrow (1)$.
Let $\E\in\K_\G$, that is, there exists $\F\subseteq C(\R,\R)$ such that $\E=E_\G(\F)$.
Let $X=\bigcup\E$.
We will show that $\E=\CL(\R)\cap\P(X)$. 

Let us take $E\in\CL(\R)\cap\P(X)$ and an arbitrary $f\in\F$.
For every $x\in E$ we have $\{x\}\in\E$, hence $f(x)\in\G[x]$.
It follows that $f\r E\subseteq\bigcup\G$.
Since $\G$ is connected, by Lemma~\ref{lem-func-conn-2}
there exists $g\in C(\R,\R)$ such that $f\r E=g\r E$ and $g\subseteq\bigcup\G$.
Since $\G$ is complete, we have $g\in\G$.
This shows that $E\in E_\G(\F)$, so $\CL(\R)\cap\P(X)\subseteq\E$.
The opposite inclusion is clear.
\end{proof}

From Theorems \ref{thm-strictly-below-equiv-1} and \ref{thm-strictly-below-equiv-2}
we obtain the following characterization.

\begin{corollary} \label{cor-strictly-below-together}
Let $\G\subseteq C(\R,\R)$ be nonempty.
The following conditions are equivalent.
\begin{enumerate}[\rm (1)]
\item $\K_\G=\{\CL(\R)\cap\P(X)\!:X\subseteq\R\}$.
\item $\G$ is a complete and connected family, and for every $x\in\R$
	there exists a function $f\in C(\R,\R)$ such that $f\setminus\bigcup\G=f\r\{x\}$.
\end{enumerate}
\end{corollary}

\section{Results for family $(-\infty,g)\cup(g,\infty)$}

For $f,g\in C(\R,\R)$, if $f(x)\neq g(x)$ for every $x$ then either $f<g$ or $f>g$.
Hence, $\{f\in C(\R,\R)\!:(\forall x\in\R)\,f(x)\neq g(x)\}=(-\infty,g)\cup(g,\infty)$.

Recall that a family $\X$ of subsets of a topological space is said to be separated
if for every distinct $X,Y\in\X$ there exist disjoint open sets $U,V$ such that
$X\subseteq U$, $Y\subseteq V$ and $(\forall Z\in\X)\ Z\subseteq U\,\lor\,Z\subseteq V$.

\begin{theorem} \label{thm-diff-equal}
Let $g\in C(\R,\R)$, $\G=\{f\in C(\R,\R)\!:(\forall x\in\R)\,f(x)\neq g(x)\}$.
Then $$\K_\G=\left\{\CL(\R)\cap\bigcup_{X\in\X}\P(X)\!:\X\subseteq\P(\R)\text{ is separated\/}\right\}.$$
\end{theorem}

\begin{proof}
Let $\E\in\K_\G$, that is, $\E=E_\G(\F)$ for some $\F\subseteq C(\R,\R)$.
For $x\in\bigcup\E$, denote $\E_x=\{E\in\E\!:x\in E\}$,
and let $\X=\big\{\bigcup\E_x\!:x\in\bigcup\E\big\}$.
We will show that $\X$ is separated and $\E=\CL(\R)\cap\bigcup_{X\in\X}\P(X)$.
Let us note that for every $x,y\in\R$, $\{x,y\}\in\E$ if and only if for every $f\in\F$,
$$(f(x)>g(x)\,\land\,f(y)>g(y))\,\lor\,(f(x)<g(x)\,\land\,f(y)<g(y)).$$
Hence, the relation $\sim$, defined by $x\sim y\Leftrightarrow\{x,y\}\in\E$, is an equivalence relation on $\bigcup\E$,
and $\X$ is the corresponding partition of $\bigcup\E$ into equivalence classes.

Let $X=\bigcup\E_x$ and $Y=\bigcup\E_y$ be distinct elements of $\X$.
Then $\{x,y\}\notin\E$, hence there exists $f\in\F$ such that $(f(x)-g(x))(f(y)-g(y))\le 0$.
We have $\{x\},\{y\}\in\E$, so $(f(x)-g(x))(f(y)-g(y))\neq 0$.
Without a loss of generality we may assume that $f(x)<g(x)$ and $f(y)>g(y)$.
Let $U=\{x\in\R\!:f(x)<g(x)\}$ and $V=\{x\in\R\!:f(x)>g(x)\}$.
Then $U,V$ are disjoint open sets such that $X\subseteq U$, $Y\subseteq V$.
Also, for every $z\in\bigcup\E$ we have $f(z)\neq g(z)$, hence $z\in U$ or $z\in V$.
Clearly, $z\in U$ implies $\E_z\subseteq U$, and similarly $z\in V$ implies $\E_z\subseteq V$,
hence the family $\X$ is separated.

We show that $\E=\CL(\R)\cap\bigcup_{X\in\X}\P(X)$.
The inclusion from left to right follows from the definition of $\X$.
Conversely, if $E\in\CL(\R)\cap\P(X)$ for some $X\in\X$ then we have $x\sim y$ for all $x,y\in E$,
hence for every $f\in\F$ we have either $f<_E g$ or $g<_E f$,
where $f<_E g$ is a shorthand for $(\forall x\in E)\,f(x)<g(x)$.
It follows that $f\r E\in\G\r E$, thus $E\in E_\G(\F)=\E$.

We have proved that $\K_\G\subseteq\big\{\CL(\R)\cap\bigcup_{X\in\X}\!:\X\subseteq\P(\R)\text{ is separated\/}\big\}$.
Let $\E=\CL(\R)\cap\bigcup_{X\in\X}\P(X)$
for some separated family $\X\subseteq\P(\R)$, and let $\F=F_\G(\E)$.
Then $f\in\F$ if and only if $f\in C(\R,\R)$ and $f\r E\in\G\r E$ for every $E\in\E$.
Hence, $\F=\{f\in C(\R,\R)\!:(\forall X\in\X)\,f<_X g\,\lor\,g<_X f\}$.

Let us further show that $E_\G(\F)=\E$.
Assume that $E\in\CL(\R)$ and $E\notin\E$.
Then either $E\nsubseteq\bigcup\X$ or there exist distinct sets $X,Y\in\X$ such that $E$ intersects both of them.
In the first case we take $z\in E\setminus\bigcup\X$ and $f\in C(\R,\R)$ such that $f(z)=g(z)$
and $f(x)>g(x)$ for all $x\neq z$.
Then $f\in\F$ but $f(z)\notin\G[z]$, hence $E\notin E_\G(\F)$.
In the second case let $U,V$ be disjoint open sets such that $X\subseteq U$, $Y\subseteq V$ and
$(\forall Z\in\X)\ Z\subseteq U\,\lor\,Z\subseteq V$.
There exists $f\in C(\R,\R)$ such that $U=\{x\in\R\!:f(x)>g(x)\}$ and $V=\{x\in\R\!:f(x)<g(x)\}$.
We have $f\in\F$ but $E\notin E_\G(\F)$.
It follows that $E_\G(\F)\subseteq\E$, hence the equality holds true, and thus $\E\in\K_\G$.
Hence, $\big\{\CL(\R)\cap\bigcup_{X\in\X}\P(X)\!:\X\subseteq\P(\R)\text{ is separated}\big\}\subseteq\K_\G$.
\end{proof}

If a family $\X\subseteq\P(\R)$ is separated then for every distinct sets $X,Y\in\X$
one can find regular open sets $U,V$ such that $X\subseteq U$, $Y\subseteq V$
and $(\forall Z\in\X)\ Z\subseteq U\,\lor\,Z\subseteq V$.
Indeed, if $U,V$ are disjoint open sets then their regularizations
$U'=\int(\cl U)$, $V'=\int(\cl V)$ satisfy $U\subseteq U'$, $V\subseteq V'$ and are disjoint as well.
We may also take $\R\setminus\cl U'$ instead of $V'$.

Let us recall that for $x\in\R$ we have denoted $\A_x=\CL(\R)\cap\P(\R\setminus\{x\})$.
For every open set $U\subseteq\R$ we also denote $\B_U=\CL(\R)\cap(\P(U)\cup\P(\R\setminus\cl U))$.

\begin{theorem} \label{thm-diff-2}
Let $\G\subseteq C(\R,\R)$ be nonempty.
The following conditions are equivalent.
\begin{enumerate}[\rm (1)]
\item $\big\{\CL(\R)\cap\bigcup_{X\in\X}\P(X)\!:\X\subseteq\P(\R)\text{ is separated\/}\big\}\subseteq\K_\G$.
\item $\{\A_x\!:x\in\R\}\cup\{\B_U\!:U\subseteq\R\text{ is regular open\/}\}\subseteq\K_\G$.
\item For any $x\in\R$ there exists $f\in C(\R,\R)$ such that $E_\G(\{f\})=\A_x$.
  Moreover, for any $x,y\in\R$ and any regular open set $U\subseteq\R$
  such that $x\in U$ and $y\notin\cl U$ there exists $f\in F_\G(\B_U)$ such that
  $f\r\{x,y\}\notin\G\r\{x,y\}$.
\end{enumerate}
\end{theorem}

\begin{proof}
$(1) \Rightarrow (2)$ is obvious.

$(2) \Rightarrow (3)$.
The first part of (3) follows from Theorem~\ref{thm-strictly-below-equiv-1}.
For the second part, let $U$ be a regular open set such that $x\in U$ and $y\notin\cl U$.
By (2), we have $\B_U\in\K_\G$, hence $E_\G(F_\G(\B_U))=\B_U$.
Since $\{x,y\}\notin\B_U$, there exists $f\in F_\G(\B_U)$ such that $f\r\{x,y\}\notin\G\r\{x,y\}$.

$(3) \Rightarrow (1)$.
Let $\X\subseteq\P(\R)$ be a separated family and let $\E=\CL(\R)\cap\bigcup_{X\in\X}\P(X)$.
Denote $\F=F_\G(\E)$.
We show that $\E=E_\G(\F)$.
Let us take $E\in\CL(\R)$, $E\notin\E$.
Then either there exists $x\in E\setminus\bigcup\X$, or there exist $x,y\in E$ and distinct sets $X,Y\in\X$
such that $x\in X$ and $y\in Y$.

If $x\in E\setminus\bigcup\X$ then let $f\in F_\G(\A_x)$ be such that $f(x)\notin\G_x$.
Since $\E\subseteq\A_x$, we have $F_\G(\A_x)\subseteq F_\G(\E)$, hence $f\in\F$.
It follows that $\{x\}\notin E_\G(\F)$, hence $E\notin E_\G(\F)$.
If $x\in X$, $y\in Y$ for some distinct $X,Y\in\X$ then there exists an open regular set $U$
such that $X\subseteq U$, $Y\subseteq\R\setminus\cl U$ and each $Z\in\X$ is covered either by $U$
or by $\R\setminus\cl U$.
It follows that $\E\subseteq\B_U$.
By (3), there exists $f\in F_\G(\B_U)$ such that $f\r\{x,y\}\notin\G\r\{x,y\}$.
We have $f\in\F$, hence $\{x,y\}\notin E_\G(\F)$, so $E\notin E_\G(\F)$.
In both cases it follows that $E_\G(\F)=\E$, hence $\E\in\K_\G$.
\end{proof}

Let us note that the family $\H=\{\A_x\!:x\in\R\}\cup\{\B_U\!:U\text{ is regular open}\}$
in the second condition of Theorem~\ref{thm-diff-2} is not minimal.
Indeed, $\B_\emptyset=\B_\R=\CL(\R)$ is an element of $\K_\G$ for every $\G\subseteq C(\R,\R)$,
hence $\H\setminus\{\B_\emptyset\}\subseteq\K_\G\,\Leftrightarrow\,\H\subseteq\K_\G$
holds for every $\G$.
We do not know whether there exists a regular open set $U$ and
a family $\G\subseteq C(\R,\R)$ such that $\H\setminus\{\B_U\}\subseteq\K_\G$ and $\B_U\notin\K_\G$.
We also do not know whether one can find a minimal family $\M$
of nonempty hereditary families of closed sets
having the property that for every $\G$, if $\M\subseteq\K_\G$ then
$\CL(\R)\cap\bigcup_{X\in\X}\P(X)\in\K_\G$ for every separated family $\X$.

To characterize families $\G$ such that
$\CL(\R)\cap\bigcup_{X\in\X}\P(X)\in\K_\G$ holds for every separated family $\X$,
we need few more notions.
Given a fixed family $\G\subseteq C(\R,\R)$ and points $a=(a_1,a_2)$, $b=(b_1,b_2)$ in $\R^2$,
let us write $a\sim b$ if there exists a function $f\in\G$ such that $f(a_1)=a_2$ and $f(b_1)=b_2$.
Clearly, if $a\sim b$ and $a_1=b_1$ then $a=b$.
We say that family $\G$ is \emph{transitive} if for any points $a,b,c\in\R^2$ having distinct
first coordinates, if $a\sim b$ and $b\sim c$ then $a\sim c$.
We say that family $\G$ is \emph{sequential} if $a\sim b$ holds true whenever $a,b\in\bigcup\G$
and there exists a sequence of points $\{a_n\!:n\in\omega\}$ in $\bigcup\G$
such that $a_n\to a$, $a_n\sim b$, and the first coordinates of points $a$, $b$, and $a_n$
are distinct, for every $n$.

Let $(I,<)$ be a linearly ordered set, and for every $i\in I$, let $\G_i\subseteq C(\R,\R)$
be nonempty.
We say that indexed system $\{\G_i\!:i\in I\}$ is \emph{sliced} if for every $i\in I$
there exist functions $g^{-}_i,g^{+}_i\in C(\R,\R^*)$ such that
$$\bigcup_{j<i}\G_j\subseteq(-\infty,g^{-}_i),\quad
\G_i\subseteq(g^{-}_i,g^{+}_i)\quad\text{and}\quad
\bigcup_{j>i}\G_j\subseteq(g^{+}_i,\infty).$$

\begin{lemma} \label{lem-sliced-1}
Let $\{\G_i\!:i\in I\}$ be a sliced system.
Then for each $i\in I$ there exist functions $h^{-}_i,h^{+}_i\in C(\R,\R^*)$ such that
$\G_i\subseteq(h^{-}_i,h^{+}_i)$ and $h^{+}_i\le h^{-}_j$ whenever $i<j$.
\end{lemma}

\begin{proof}
For every $i$, let $g^{-}_i,g^{+}_i\in C(\R,\R^*)$ be such that
$\bigcup_{j<i}\G_j\subseteq(-\infty,g^{-}_i)$, $\G_i\subseteq(g^{-}_i,g^{+}_i)$,
and $\bigcup_{j>i}\G_j\subseteq(g^{+}_i,\infty)$.
It is clear that if $\bigcup_{j<i}\G_j\neq\emptyset$ then $g^{-}_i\in C(\R,\R)$
and, similarly, if $\bigcup_{j>i}\G_j\neq\emptyset$ then $g^{+}_i\in C(\R,\R)$.
Since each interval $(g^{-}_i(0),g^{+}_i(0))$ contains a rational number,
$I$ is at most countable.

For simplicity let us assume that $I$ is infinite.
In the finite case the proof will be the same.
Let $\{i(n)\!:n<\omega\}$ be a one-to-one enumeration of $I$.
By induction let us define
\begin{align*}
h^{-}_{i(n)}(x)=\max\big(\big\{h^{+}_{i(m)}(x)\!:m<n&\,\land\,i(m)<i(n)\big\}\cup\big\{g^{-}_{i(n)}(x)\big\}\big),\\
h^{+}_{i(n)}(x)=\min\big(\big\{h^{-}_{i(m)}(x)\!:m<n&\,\land\,i(m)>i(n)\big\}\cup\big\{g^{+}_{i(n)}(x)\big\}\big).
\end{align*}

It is clear that $h^{-}_i,h^{+}_i\in C(\R,\R^*)$, and $\G_i\subseteq(h^{-}_i,h^{+}_i)$.
Moreover, if $m<n$ then either $i(m)<i(n)$ and then $h^{+}_{i(m)}\le h^{-}_{i(n)}$ by the definition
of $h^{-}_{i(n)}$, or $i(m)>i(n)$ and then $h^{+}_{i(n)}\le h^{-}_{i(m)}$ by the definition of
$h^{+}_{i(n)}$.
Hence, $h^{+}_i\le h^{-}_j$ for any $i<j$.
\end{proof}

\begin{lemma} \label{lem-sliced-2}
Let $\G\subseteq C(\R,\R)$ be a complete, transitive and sequential family.
Then there exists a sliced system $\{\G_i\!:i\in I\}$ such that $\G=\bigcup_{i\in I}\G_i$ and
each $\G_i$ is complete and connected.
\end{lemma}

\begin{proof}
We may assume that $\G$ is nonempty.
Denote $H=\bigcup\G$.
For $a,b\in\R^2$, let us write $a\approx b$ if there exists $c$ such that $a\sim c\sim b$.
We prove that $\approx$ is an equivalence relation on $H$.
The symmetry and the reflexivity of $\approx$ is clear.
For the transitivity it suffices to prove that $a\sim b\sim c\sim d$ implies $a\approx d$.
Let $a=(a_1,a_2)$, $b=(b_1,b_2)$, $c=(c_1,c_2)$, $d=(d_1,d_2)$.
We may assume that $a\neq b\neq c\neq d$, hence $a_1\neq b_1\neq c_1\neq d_1$.
If $a_1\neq c_1$ then by transitivity of $\G$ we have $a\sim c$, and we are done.
A similar argument works if $b_1\neq d_1$, so we may assume that $a_1=c_1$ and $b_1=d_1$.
Without a loss of generality, let $a_1<b_1$.
Let $f\in\G$ be such that $f(b_1)=b_2$, and let $b'=(b'_1,b'_2)$ be such that
$b'_1>b_1$ and $b'_2=f(b'_1)$.
Then we have $a\sim b'$ and $b'\sim c$.
Since $b'_1\neq d_1$, it follows that $b'\sim d$, and thus $a\approx d$.

Let $\{H_i\!:i\in I\}$ be the partition of $H$ corresponding to the equivalence $\approx$,
and for every $i\in I$ let $\G_i=\{f\in\G\!:f\subseteq H_i\}$.
Let $f\in\G$ be arbitrary.
For all points $a,b\in f$ we have $a\approx b$, hence $f\subseteq H_i$ for some $i$.
It follows that $\G=\bigcup\{\G_i\!:i\in I\}$.
By the definition of $\approx$ and the completeness of $\G$, each $\G_i$ is connected and complete,
and we have $\bigcup\G_i=H_i$.
Clearly, if $\G_i\neq\G_j$ then 
$(\forall f\in\G_i)(\forall g\in\G_j)\,f<g$ or $(\forall f\in\G_i)(\forall g\in\G_j)\,f>g$.
Thus there exists a linear order on $I$ so that
$i<j$ if and only if $f<g$ for all $f\in\G_i$ and $g\in\G_j$.

For every $i$ such that $\bigcup_{j>i}\G_j\neq\emptyset$,
let us define $h_1(x)=\inf\{\sup A_{x,\varepsilon}\!:\varepsilon>0\}$ and
$h_2(x)=\sup\{\inf B_{x,\varepsilon}\!:\varepsilon>0\}$,
where $A_{x,\varepsilon}=H_i\cap\big((x-\varepsilon,x+\varepsilon)\times\R\big)$ and
$B_{x,\varepsilon}=\bigcup_{j>i}H_j\cap\big((x-\varepsilon,x+\varepsilon)\times\R\big)$.
Then $h_1$ is upper semi-continuous, $h_2$ is lower semi-continuous,
and we have $f\le h_1\le h_2\le g$ for all $f\in\G_i$ and $g\in\bigcup_{j>i}\G_j$.
For $x\in\R$ denote $a=(x,h_1(x))$ and assume that $a\in H$.
Then there exists a sequence $\{a_n\!:n\in\omega\}$ in $H_i$ converging to $a$
and such that each $a_n$'s first coordinate is distinct from $x$.
Let $b\in H_i$ be such that for every $n$, first coordinates of $a$, $b$, and $a_n$ are distinct.
We have $a_n\sim b$ for every $n$.
Since $\G$ is sequential, it follows that $a\sim b$, hence $a\in H_i$.
Similarly, if $b=(x,h_2(x))$ and $b\in H$ then there exists $k\in I$ such that
$k=\min\{j\in I\!:j>i\}$, and $b\in H_k$.
It follows that if $h_1(x)=h_2(x)=y$ then $(x,y)\notin H$.

By a theorem of Michael (see~\cite{Engelking}, Exercise 1.7.15 (d)),
there exists a continuous function $h^{+}\colon\R\to\R$ such that $h_1\le h^{+}\le h_2$
and for every $x\in\R$, if $h_1(x)<h_2(x)$ then $h_1(x)<h^{+}(x)<h_2(x)$.
It follows that $f<h^{+}<g$ for any $f\in\G_i$ and $g\in\bigcup_{j>i}\G_j$.
A similar argument shows that if $\bigcup_{j<i}\G_j\neq\emptyset$ then
there exists $h^{-}\in C(\R,\R)$ such that $f<h^{-}<g$ for any $f\in\bigcup_{j<i}\G_j$ and $g\in\G_i$.
Hence, $\{\G_i\!:i\in I\}$ is a sliced system.
\end{proof}

\begin{theorem} \label{thm-diff-3}
Let $\G\subseteq C(\R,\R)$ be nonempty.
Then the following conditions are equivalent.
\begin{enumerate}[\rm (1)]
\item $\K_\G\subseteq\big\{\CL(\R)\cap\bigcup_{X\in\X}\P(X)\!:\X\text{ is separated\/}\big\}$.
\item There exists a sliced system $\{\G_i\!:i\in I\}$ of complete connected families
  such that $\G=\bigcup_{i\in I}\G_i$.
\end{enumerate}
\end{theorem}

\begin{proof}
$(1) \Rightarrow (2)$.
Denote $H=\bigcup\G$.
If $g\in C(\R,\R)$ and $g\subseteq H$ then let $\E=E_\G(\{g\})$.
Since $\E\in\K_\G$, there exists a strictly separated family $\X\subseteq\P(\R)$ such that
$\E=\CL(\R)\cap\bigcup_{X\in\X}\P(X)$.
For every $x\in\R$ we have $g(x)\in\G_[x]$, hence $\{x\}\in\E$, and thus $\bigcup\X=\R$.
For any disjoint open sets $U,V\subseteq\R$, if $U\cup V=\R$ then either $U=\R$ or $V=\R$.
It follows that $\X=\{\R\}$, hence $\R\in\E$ and thus $g\in\G$.
Hence, $\G=\{g\in C(\R,\R)\!:g\subseteq H\}$, so $\G$ is complete.

Let us show that $\G$ is transitive.
Let $a=(a_1,a_2)$, $b=(b_1,b_2)$, $c=(c_1,c_2)$ be such that $a_1,b_1,c_1$ are distinct,
and let $a\sim b\sim c$.
Since for any $(x,y)\in H$ there exists $g\in\G$ such that $g(x)=y$,
we can find $g,h\in\G$ such that $g(a_1)=a_2$, $g(b_1)=h(b_1)=b_2$, and $h(c_1)=c_2$.
Let $f\in C(\R,\R)$ be any function such that $f(a_1)=a_2$, $f(b_1)=b_2$, $f(c_1)=c_2$,
and let $\E=E_\G(\{f\})$.
Then $\{a_1,b_1\}\in\E$, $\{b_1,c_1\}\in\E$,
hence also $\{a_1,c_1\}\in\E$, and it follows that
there exists $f'\in\G$ such that $f'(a_1)=a_2$ and $f'(c_1)=(c_2)$.
Since $f'\subseteq H$, we have $f'\in\G$, hence $a\sim c$.

To show that $\G$ is also sequential, assume that $a,b\in H$ and there exists a sequence
$\{a_n\!:n\in\omega\}$ in $H$ such that $a_n\to a$ , $a_n\sim b$, and first coordinates
of points $a$, $b$, and $a_n$ are distinct, for every $n$.
We have to prove that $a\sim b$.
There exists a function $f\in C(\R,\R)$ such that $f(x)=y$, $f(u)=v$, and $f(x_n)=y_n$ for all $n$,
where $(x,y)=a$, $(u,v)=b$, and $(x_n,y_n)=a_n$.
Denote $\E=E_\G(\{f\})$.
Since $\E\in\K_\G$, by (1) there exists a separated family $\X\subseteq\P(\R)$
such that $\E=\CL(\R)\cap\bigcup_{X\in\X}\P(X)$.
We have $\{x\}\in\E$ and $\{x_n,u\}\in\E$, for every $n$.

It suffices to show that $\{x,u\}\in\E$.
If this is not the case then there exist distinct sets $X,Y\in\X$ such that $x\in X$ and $u\in Y$.
Since $\X$ is separated, there exist disjoint open sets $U,V$ such that $X\subseteq U$, $Y\subseteq V$,
and $(\forall Z\in\X)\ Z\subseteq U\,\lor\,Z\subseteq V$.
Since $x_n\to x$, there exists $n$ such that $x_n\in U$.
But then $x_n\notin Y$, and this is in contradiction with $\{x_n,u\}\in\E$.

We have proved that $\G$ is complete, transitive and sequential.
Then condition~(2) follows from Lemma~\ref{lem-sliced-2}.

$(2) \Rightarrow (1)$.
Let $\{\G_i\!:i\in I\}$ be a sliced system of complete connected families such that
$\G=\bigcup_{i\in I}\G_i$.
By Lemma~\ref{lem-sliced-1}, for every $i$ there exist $h^{-}_i,h^{+}_i\in C(\R,\R^*)$ such that
$\G_i\subseteq(h^{-}_i,h^{+}_i)$ and $h^{+}_i\le h^{-}_j$ whenever $i<j$.

Let $f\in C(\R,\R)$ be arbitrary and let $\E=E_\G(\{f\})$.
For every $i$, denote $X_i=\{x\in\R\!:f(x)\in\G_x\,\land\,h^{-}_i(x)<f(x)<h^{+}_{i}(x)\}$,
and let $\X=\{X_i\!:i\in I\}$.
For every $i\in I$, let us take $U_i=\{x\in\R\!:h^{-}_i(x)<f(x)<h^{+}_i(x)\}$ and
$V_i=\{x\in\R\!:f(x)<h^{-}_i(x)\,\lor\,h^{+}_i(x)<f(x)\}$.
Then $U_i,V_i$ are disjoint open sets such that $X_i\subseteq U_i$ and
$X_j\subseteq V_i$ for every $j\neq i$.
It follows that $\X$ is a separated family.

We will prove that $\E=\CL(\R)\cap\bigcup_{X\in\X}\P(X)$.
If $E\in\E$ then there exists $g\in\G$ such that $f\r E=g\r E$.
We have $g\in\G_i$ for some $i\in I$, and it easy to see that $E\subseteq X_i$.
We can conclude that $\E\subseteq\CL(\R)\cap\bigcup_{X\in X}\P(X)$.

To prove the opposite inclusion, assume that $E\notin\E$,
hence for every $g\in\G$ we have $f\r E\neq g\r E$.
If there exists $x\in E$ such that $f(x)\notin\G_x$ then $x\notin\bigcup\X$, and hence
$E\notin\CL(\R)\cap\bigcup_{X\in\X}\P(X)$.
Assume further that $E\subseteq\{x\in\R\!:f(x)\in\G_x\}$.
If there exists $i\in I$ such that $f\r E\subseteq\bigcup\G_i$, then by Lemma~\ref{lem-func-conn-2}
there exists $g\in\G_i$ such that $f\r E=g\r E$, which is impossible.
Hence, there exist $i\neq j$ and $x,y\in E$ such that
$(x,f(x))\in\bigcup\G_i$ and $(y,f(y))\in\bigcup\G_j$.
It follows that $x\in X_i$ and $y\in X_j$, hence $E\notin\CL(\R)\cap\bigcup_{X\in\X}\P(X)$.

We have proved that for every $f\in C(\R,\R)$ there exists a separated family $\X_f\subseteq\P(\R)$
such that $E_\G(\{f\})=\CL(\R)\cap\bigcup_{X\in\X_f}\P(X)$.
For arbitrary $\F\subseteq C(\R,\R)$, we have $E_\G(\F)=\bigcap_{f\in\F}E_\G(\{f\})$.
Let us take
$$\X=\Bigg\{\bigcap_{f\in\F}X_f\!:\langle X_f\!:f\in\F\rangle\in\prod_{f\in\F}\X_f\,\land\,
\bigcap_{f\in\F}X_f\neq\emptyset\Bigg\}.$$

We show that $\X$ is separated.
Let $\langle X_f\!:f\in\F\rangle,\langle Y_f\!:f\in\F\rangle\in\prod_{f\in\F}X_f$
be such that $\bigcap_{f\in\F}\neq\emptyset$, $\bigcap_{f\in\F}Y_f\neq\emptyset$,
and $\bigcap_{f\in\F}\neq\bigcap_{f\in\F}Y_f$.
Then there exists $h\in\F$ such that $X_h\neq Y_h$.
Since $\X_h$ is separated, there exist disjoint open sets $U,V\subseteq\R$ such that
$X_h\subseteq U$, $Y_h\subseteq V$, and $(\forall Z\in\X_h)\ Z\subseteq U\,\lor\,Z\subseteq V$.
It follows that $\bigcap_{f\in\F}X_f\subseteq U$ and $\bigcap_{f\in\F}Y_f\subseteq V$.
For every $\langle Z_f\!:f\in\F\rangle\in\prod_{f\in\F}\X_f$ we have
$Z_h\subseteq U\,\lor\,Z_h\subseteq V$, hence also
$\bigcap_{f\in\F}Z_f\subseteq U\,\lor\,\bigcap_{f\in\F}Z_f\subseteq V$.

For $E\in\CL(\R)$, we have $E\in E_\G(\F)$ if and only if
$(\forall f\in\F)(\exists X\in\X_f)\,E\subseteq X$ if and only if
$(\exists X\in\X)\,E\subseteq X$.
Hence, $E_\G(\F)=\CL(\R)\cap\bigcup_{X\in\X}\P(X)$ and condition~(1) follows.
\end{proof}

\end{document}